\documentclass[11pt,reqno]{amsart}

\usepackage{amscd,amssymb,amsmath,amsthm}
\usepackage[arrow,matrix]{xy}
\usepackage{graphicx}
\usepackage{color}
\usepackage{cite}
\newtheorem{thm}{Theorem}
\newtheorem{defn}{Definition}
\newtheorem{lemma}{Lemma}
\newtheorem{pro}{Proposition}
\newtheorem{rk}{Remark}
\newtheorem{cor}{Corollary}

\newtheorem{ex}{Example}

\numberwithin{equation}{section} 

\newcommand{\A}{{\mathcal A}}
\newcommand{\B}{{\mathcal B}}

\newcommand{\bea}{\begin{eqnarray}}
\newcommand{\eea}{\end{eqnarray}}

\newcommand{\R}{\mathbb{R}}
\newcommand{\C}{\mathbb{C}}

\def\R{\mathbb{R}}

\def\C{\mathcal{C}}

\DeclareMathOperator{\rank}{rank}
\DeclareMathOperator{\Hom}{Hom}

\begin{document}

\title[Evolution algebra of a ``chicken'' population]{Evolution algebra of a ``chicken'' population}

\author{M. Ladra,  U. A. Rozikov}

 \address{M.~Ladra\\Departamento de \'Algebra, Universidad de Santiago de Compostela\\15782
Santiago de Compostela, Spain}
\email {manuel.ladra@usc.es}

 \address{U.\ A.\ Rozikov\\ Institute of mathematics,
29, Do'rmon Yo'li str., 100125, Tashkent, Uzbekistan.}
\email {rozikovu@yandex.ru}

\begin{abstract} We consider an evolution algebra which corresponds to a
bisexual population with a set of
females partitioned into finitely many different types
and the males having only one type. We study basic properties of the algebra.
This algebra is commutative (and hence flexible), not
associative and not necessarily power associative, in general. Moreover it is not unital. A condition is found 
on the structural constants of the algebra under which the algebra is associative,
alternative, power associative, nilpotent, satisfies Jacobi and Jordan identities. In a general case, we
describe the full set of idempotent elements and the full set of
absolute nilpotent elements. The set of all operators of left (right) multiplications is described. Under some conditions on the structural constants it is proved that the corresponding algebra is centroidal.  Moreover the classification of 2-dimensional and some 3-dimensional algebras are obtained.
\end{abstract}
\maketitle



\section{Introduction} \label{sec:intro}

Description of a sex linked inheritance
with algebras involves overcoming the obstacle of asymmetry in the genetic
inheritance rules. Inheritance which is not sex linked is symmetrical with respect
to the sexes of the organisms \cite{ly}, while sex linked inheritance is not (see \cite{LR},\cite{m}).
The main problem for a given algebra of a sex linked population is to carefully examine how
the basic algebraic model must be
altered in order to compensate for this lack of symmetry in the genetic inheritance
system. In \cite{e3}  Etherington began the study of this kind of algebras with the
simplest possible case.

Now the methods of mathematical genetics have become probability theory, stochastic processes, nonlinear differential and difference equations and non-associative algebras.  The book \cite{ly} describes some mathematical apparats of studying  algebras of genetics.  This book mainly considers a {\it free population}, which means random mating in the population.
Evolution of a free population can be given by a dynamical system generated by a quadratic stochastic
operator (QSO) and by an evolution algebra of a free population.
In \cite{ly} an evolution algebra associated to the
free population is introduced and using this non-associative algebra,
many results are obtained in explicit form, e.g., the explicit
description of stationary quadratic operators, and the explicit
solutions of a nonlinear evolutionary equation in the absence of
selection, as well as general theorems on convergence to equilibrium
in the presence of selection. In \cite{GMR} some recently obtained
results and also several open problems related
to the theory of QSOs are discussed. See also \cite{ly} for more detailed theory of QSOs.

Recently in \cite{LR} an evolution algebra $\B$ is introduced identifying the coefficients
of inheritance of a bisexual population as the structure constants
of the algebra.  The basic properties of the algebra are studied.
Moreover a detailed analysis of a special case
of the evolution algebra (of bisexual population in which type ``1''
of females and males have preference) is given. Since the structural constants of the algebra $\B$ are given by two cubic matrices, the study of this algebra is difficult. To avoid such difficulties one has to consider an algebra of bisexual population with a simplified form of matrices of structural constants.
In this paper we consider a such simplified model of bisexual population and study corresponding evolution algebra.

The paper is organized as follows. In Section 2 we define our algebra as an evolution algebra which corresponds to a
bisexual population with a set of
females partitioned into finitely many different types
and the males having only one type. Then we study basic properties (associativity, non-associativity, commutativity, power associativity, nilpotency, unitality,    etc.) of the algebra. Section 3 is devoted to subalgebras, absolute nilpotent elements and idempotent elements of the algebra. In Section 4 the set of all operators of left (right) multiplications is described. In Section 5, under some conditions on the structural constants, it is proved that the corresponding algebra is centroidal. The last section gives a classification of 2-dimensional and some 3-dimensional algebras.

\section{Definition and basic properties of the EACP}

We consider a set  $\{h_i, i=1,\dots,n\}$ (the set of ``hen"s) and $r$ (a ``rooster").

\begin{defn} Let $(\mathcal C, \cdot)$ be an algebra over a field $K$ (with characteristic $\ne 2$). If it admits a basis $\{h_1,\dots,h_n, r\}$, such that
\begin{equation}\label{4}
\begin{array}{ll}
h_ir = rh_i=\frac{1}{2} \left(\sum_{j=1}^na_{ij}h_j+ b_ir\right),  \\
h_ih_j =0,\ \ i,j=1,\dots,n; \ \ rr =0, \end{array} \end{equation}
then this algebra is called an evolution algebra of a ``chicken" population (EACP).
We call the basis $\{h_1,\dots,h_n, r\}$ a natural basis.
\end{defn}

\begin{rk} If
\begin{equation}\label{xh}
\sum_{j=1}^na_{ij}=1; \ \ b_i=1; \ \ \mbox{for all} \ \ i=1,2,\dots,n
 \end{equation}
 then the corresponding $\mathcal C$ is a particular case of an evolution algebra of a bisexual population, $\mathcal B$, introduced in \cite{LR}. The study of the algebra $\mathcal B$ is difficult, since it is determined by two cubic matrices. While the algebra $\mathcal C$ is more simpler, since it is defined by a rectangular $n\times (n+1)$-matrix
$$M=\left(\begin{array}{ccccc}
a_{11}&a_{12}&\dots&a_{1n}&b_1\\[2mm]
a_{21}&a_{22}&\dots&a_{2n}&b_2\\[2mm]
\vdots&\vdots&\vdots&\vdots&\vdots\\[2mm]
a_{n1}&a_{n2}&\dots&a_{nn}&b_n
\end{array}\right),$$
which is called the matrix of structural constants of the algebra $\mathcal C$.
This simplicity allows to obtain deeper results on $\mathcal C$ than on $\mathcal B$. Moreover, in this paper we do not require the condition (\ref{xh}).
\end{rk}

The general formula for the multiplication is the extension of
(\ref{4}) by bilinearity, i.e. for $x,y\in \C$,
$$ x=\sum_{i=1}^nx_ih_i+ur, \ \  y=\sum_{i=1}^ny_ih_i+vr$$
using (\ref{4}), we obtain
\begin{equation}\label{xy}
xy={1\over 2}\sum_{j=1}^n\left(\sum_{i=1}^n(vx_i+uy_i)a_{ij}\right)h_j+{1\over 2}\left(\sum_{i=1}^n(vx_i+uy_i)b_i\right)r
\end{equation}
and
\begin{equation}\label{xx}
x^2=xx=\sum_{j=1}^n\left(\sum_{i=1}^n(ux_i)a_{ij}\right)h_j+\left(\sum_{i=1}^n(ux_i)b_i\right)r.
\end{equation}

We recall the following definitions:
If $x, y$ and $z$ denote arbitrary elements of an algebra then

Associative: $(xy)z = x(yz)$.

Commutative: $xy = yx$.

Anticommutative: $xy = -yx$.

Jacobi identity: $(xy)z + (yz)x + (zx)y = 0$.

Jordan identity: $(xy)x^2 = x(yx^2)$.

Power associative: For all $x$, any three nonnegative powers of $x$ associate. That is if $a, b$ and $c$ are nonnegative powers of $x$, then $a(bc) = (ab)c$. This is equivalent to saying that $x^m x^n = x^{n+m}$ for all nonnegative integers $m$ and $n$.

Alternative: $(xx)y = x(xy)$ and $(yx)x = y(xx)$.

Flexible: $x(yx) = (xy)x$.

It is known that these properties are related by
\begin{itemize}
\item {\it associative} implies {\it alternative} implies {\it power associative};

\item {\it associative} implies {\it Jordan identity} implies {\it power associative};

\item Each of the properties associative, commutative, anticommutative, Jordan identity, and Jacobi identity individually imply flexible.
\end{itemize}
For a field with characteristic not two, being both commutative and anticommutative implies the algebra is just $\{0\}$.

By \cite[Theorem 4.1.]{LR} we have
\begin{itemize}
  \item[(1)] Algebra $\mathcal C$ is not associative, in
general.
  \item[(2)] Algebra $\mathcal C$ is commutative, flexible.
  \item[(3)] $\mathcal C$ is not power-associative, in general.
\end{itemize}

Now we shall give conditions on the matrix $M$ under which $\C$ will be associative.

\begin{thm} The algebra $\C$ is associative iff the elements of the corresponding matrix $M$ satisfy the following
\begin{equation}\label{ac}
\sum_{j=1}^na_{ij}a_{jk}=0, \ \ b_i=0, \ \ \mbox{for any} \ \ i,k=1,2,\dots,n.
\end{equation}
\end{thm}
\begin{proof} {\sl Necessity.} Assume the algebra $\C$ is associative. First we consider the equality $(xy)z=x(yz)$ for basis elements.
If $x,y,z\in \{h_1,\dots,h_n\}$ or $x=y=z=r$ then the equality is obvious.

From $(h_ir)h_j=h_i(rh_j)$ we get
$$a_{jk}b_i=a_{ik}b_j, \ \ \mbox{for all} \ \ i,j,k=1,\dots, n.$$

The equality   $(h_ih_j)r=h_i(h_jr)$ gives
$$a_{ik}b_j=0, \ \ b_ib_j=0\ \ \mbox{for all} \ \ i,j,k=1,\dots, n.$$

From $(rh_i)r=r(h_ir)=(h_ir)r=h_i(rr)=0$ we obtain
$$\sum_{j=1}^na_{ij}a_{jk}=0, \ \ \mbox{for all} \ \ i,k=1,\dots,n.$$
These conditions imply the condition (\ref{ac}).

{\sl Sufficiency.} Assume the condition (\ref{ac}) is satisfied, then the following lemma shows that the algebra $\C$ is associative.
\end{proof}
\begin{lemma}\label{3l} If the condition (\ref{ac}) is satisfied then
\begin{equation}\label{c0}
xyz=0, \ \ \mbox{for all} \ \ x,y,z\in \C.
\end{equation}
\end{lemma}
\begin{proof} Under condition (\ref{ac}) from (\ref{xy}) we get
\begin{equation}\label{xy0}
xy={1\over 2}\sum_{j=1}^n\left(\sum_{i=1}^n(vx_i+uy_i)a_{ij}\right)h_j.
\end{equation}
Consequently, for $z=\sum_{m=1}^nz_mh_m+wr$, using condition (\ref{ac}) we get
 $$
(xy)z={1\over 2}\sum_{j=1}^nw\left(\sum_{i=1}^n(vx_i+uy_i)a_{ij}\right)(h_jr)=$$
$${1\over 4}\sum_{m=1}^n\left(w\sum_{j=1}^n\left(\sum_{i=1}^n(vx_i+uy_i)a_{ij}\right)a_{jm}\right)h_m=$$
$${1\over 4}\sum_{m=1}^n\left(\sum_{i=1}^n\left[w(vx_i+uy_i)\left(\sum_{j=1}^na_{ij}a_{jm}\right)\right]\right)h_m=0.$$
\end{proof}
By this lemma and above mentioned properties we get
\begin{cor} If the condition (\ref{ac}) is satisfied then
algebra $\C$ is alternative, power associative, satisfies Jacobi and Jordan identities.
\end{cor}

We note that the conditions (\ref{xh}) and (\ref{ac}) cannot be satisfied simultaneously,
so the corresponding algebra $\mathcal B$ of a bisexual population is not associative.

\begin{ex} The following matrix  $M$ (for $n=2$) satisfies  the condition (\ref{ac}):
$$M=\left(\begin{array}{ccc}
a&b&0\\
c&-a&0
\end{array}
\right),
$$
for any $a,b,c$ with $a^2=-bc$.
\end{ex}

\begin{defn} An element $x$ of an algebra $\mathcal A$  is called nil if
there exists $n(a) \in \mathbb{N}$ such that $(\cdots
\underbrace{((x\cdot x)\cdot x)\cdots x}_{n(a)})=0$.
The algebra $\mathcal A$ is called nil if every  element of the algebra
is nil.
\end{defn}

For $k\geq 1$, we introduce the following sequences:
$$\mathcal A^{(1)} =\mathcal A^, \ \ \mathcal A^{(k+1)} = \mathcal A^{(k)}\mathcal A^{(k)}.$$
$$\mathcal A^{<1>} = \mathcal A, \ \ \mathcal A^{<k+1>} = \mathcal A^{<k>}\mathcal A.$$
$$\mathcal A^1 = \mathcal A, \ \ \mathcal A^k =\sum_{i=1}^{k-1}\mathcal A^i\mathcal A^{k-i}.$$

\begin{defn} An algebra $\mathcal A$ is called
\begin{itemize}
\item[(i)] solvable if there exists $n\in \mathbb{N}$ such that $\mathcal A^{(n)} = 0$ and the minimal such number is
called index of solvability;
\item[(ii)] right nilpotent if there exists $n\in \mathbb{N}$ such that $\mathcal A^{<n>} = 0$ and the minimal such
number is called index of right nilpotency;
\item[(iii)] nilpotent if there exists $n\in \mathbb{N}$ such that $\mathcal A^n = 0$ and the minimal such number
is called index of nilpotency.
\end{itemize}
\end{defn}

We note that for an EACP notions as nil, nilpotent and right nilpotent algebras
are equivalent. However, the indexes of nility, right
nilpotency and nilpotency do not coincide in general.

The following is also a corollary of Lemma \ref{3l}.

\begin{cor} If the condition (\ref{ac}) is satisfied then
algebra $\C$ is nilpotent with nilpotency index equal 3.
\end{cor}

Recall that an algebra is unital or unitary if it has an
element $e$ with $ex = x = xe$ for all $x$ in the algebra.

\begin{pro} The algebra $\C$ is not unital.
\end{pro}
\begin{proof}
Assume $e=\sum_{i=1}^n a_ih_i+br$ be a unity element.
We then have $eh_i=h_i$ which gives
\begin{equation}\label{e1}
ba_{jj}=1; \ \ ba_{jm}=0,\, m\ne j; \ \ bb_j=0, \ \ \mbox{for any} \ \ j=1,\dots,n.
\end{equation}
From $er=r$ we get
\begin{equation}\label{e2}
\sum_{i=1}^na_ia_{ij}=0, \ \ \mbox{for any} \ \ j=1,\dots,n; \ \  \sum_{i=1}^na_ib_i=1.
\end{equation}
From system (\ref{e1}) we get $b\ne 0$ and $b_i=0$ for all $i$. But for this $b_i$ the second equation of the system (\ref{e2}) is not satisfied. This completes the proof.
\end{proof}

An algebra $\mathcal A$ is a division algebra if for every $a,b\in \mathcal A$ with $a\ne 0$ the equations $ax=b$ and $xa=b$ are solvable in $\mathcal A$.
\begin{pro} The algebra $\C$ is not a division algebra.
\end{pro}
\begin{proof}
Since $\C$ is a commutative algebra we shall check only $ax=b$. For coordinates of any $a=\sum_{i=1}^n\alpha_ih_i+\alpha r$, $b=\sum_{i=1}^n\beta_ih_i+\beta r$, $x=\sum_{i=1}^nx_ih_i+ur$ the equation $ax=b$ has the following form
\begin{equation}\label{da1}
\begin{array}{ll}
\left(\sum_{i=1}^na_{ij}\alpha_i\right)u+\alpha\sum_{i=1}^n a_{ij}x_i=2\beta_j, \ \ j=1,\dots,n,\\[3mm]
\left(\sum_{i=1}^nb_i\alpha_i\right)u+\alpha\sum_{i=1}^n b_ix_i=2\beta.
\end{array}
\end{equation}
So this is a linear system with $n+1$ unknowns $x_1,\dots,x_n, u$.    This system can be written as ${\bf M}y=B$ where
$y^T=(x_1,\dots,x_n,u)$, $B=2(\beta_1,\dots,\beta_n,\beta)$ and
$${\bf M}=\alpha^n\cdot\left(\begin{array}{ccccc}
a_{11}& a_{21}&\dots& a_{n1}& \sum_{i=1}^na_{i1}\alpha_i\\[2mm]
a_{12}& a_{22}&\dots& a_{n2}& \sum_{i=1}^na_{i2}\alpha_i\\[2mm]
\vdots&\vdots&\vdots&\vdots&\vdots\\[2mm]
a_{1n}& a_{2n}&\dots& a_{nn}& \sum_{i=1}^na_{in}\alpha_i\\[2mm]
b_1& b_2&\dots& b_n& \sum_{i=1}^nb_i\alpha_i
\end{array}\right).$$

By the very known Kronecker-Capelli theorem the system of linear equations ${\bf M}y=B$ has a solution if and only if the rank of matrix ${\bf M}$ is equal to the rank of its augmented matrix $({\bf M}|B)$.
Since the last column of the matrix ${\mathbf M}$ is a linear combination of the other columns of the matrix, we have
$\det({\bf M})=0$. Consequently $\rank {\mathbf M}\leq n$.  Moreover since dimension of the algebra $\C$ is $n+1$ one can choose $b$, i.e. the vector $B$ such that $\rank ({\mathbf M}|B)=1+\rank {\mathbf M}$. Then for such $b$ the equation $ax=b$ is not solvable. This completes the proof.
\end{proof}
\section{Evolution subalgebras and operator corresponding to $\C$}

By analogues of \cite[Definition 4, p. 23]{t} we give the following

\begin{defn}\begin{itemize}
\item[1)] Let $\C$ be an EACP, and $\C_1$ be a subspace of $\C$.
If $\C_1$ has a natural basis, $\{h_1',h_2',\dots,h_m',r'\}$, with multiplication table like (\ref{4}), we call $\C_1$ an evolution subalgebra of a CP.

\item[2)] Let $I\subset \C$ be an evolution subalgebra of a CP.
If $\C I\subseteq I$, we call $I$ an evolution ideal of a CP.

\item[3)] Let $\C$ and $\mathcal D$ be EACPs, we say a linear homomorphism $f$
from $\C$ to $\mathcal D$ is an evolution homomorphism, if $f$ is an algebraic map and for
a natural basis $\{h_1,\dots, h_n, r\}$ of $\C$, $\{f(r), f(h_i), i=1,\dots,n\}$ spans an evolution subalgebra of a CP
in $\mathcal D$. Furthermore, if an evolution homomorphism is one to one and onto, it
is an evolution isomorphism.

\item[4)] An EACP, $\C$ is simple if it has no proper evolution ideals.

\item[5)] $\C$ is irreducible if it has no proper subalgebras.
\end{itemize}
\end{defn}

The following proposition gives some evolution subalgebras of a CP.

\begin{pro} Let $\C$ be an EACP with the natural basis $\{h_1,\dots,h_n,r\}$ and matrix
$$M=\left(\begin{array}{cccccc}
a_{11}&0&0&\dots&0&b_1\\[2mm]
a_{21}&a_{22}&0&\dots&0&b_2\\[2mm]
\vdots&\vdots&\vdots&\vdots&\vdots&\vdots\\[2mm]
a_{n1}&a_{n2}&a_{n3}&\dots&a_{nn}&b_n
\end{array}\right).$$
Then for each $m$, $1\leq m\leq n$, the algebra $\C_m=\langle h_1,\dots,h_m,r\rangle\subset \C$ is an evolution subalgebra of a CP.
\end{pro}
\begin{proof} For given $M$ it is easy to see that $\C_m$ is closed under multiplication. The chosen subset of the natural
basis of $\C$ satisfies (\ref{4}).

\end{proof}

The following is an example of a subalgebra
of $\C$, which is not an evolution subalgebra of a CP.

\begin{ex} Let $\C$ be EACP with basis $\{h_1,h_2,h_3,r\}$ and multiplication
defined by $h_ir=h_i+r$, $i=1,2,3$. Take $u_1=h_1+r$, $u_2=h_2+r$. Then
$$(au_1+bu_2)(cu_1+du_2)=acu_1^2+(ad+bc)u_1u_2+bdu_2^2=(2ac+ad+bc)u_1+(2bd+ad+bc)u_2.$$
Hence, $F=Ku_1+Ku_2$ is a subalgebra of $\C$, but it is not an evolution subalgebra of a CP. Indeed,
assume $v_1, v_2$ be a basis of $F$. Then $v_1=au_1+bu_2$ and $v_2=cu_1+du_2$ for some $a,b,c,d\in K$ such that
$D=ad-bc\ne 0$. We have $v_1^2=(2a^2+2ab)u_1+(2b^2+2ab)u_2$ and $v_2^2=(2c^2+2cd)u_1+(2d^2+2cd)u_2$. We must have $v_1^2=v_2^2=0$, i.e.
$$ a^2+ab=0, \ \ b^2+ab=0, \ \ c^2+cd=0, \ \ d^2+cd=0.$$

From this we get $a=-b$ and $c=-d$. Then $D=0$, a contradiction. If $a=0$ then $b=0$ (resp. $c=0$ then $d=0$), we reach the same contradiction. Hence $v_1^2\ne 0$ and $v_2^2\ne 0$, and consequently $F$ is not an evolution subalgebra of a CP.
\end{ex}

Let $\C$ be an EACP on the field $K=\R$, with a basis set $\{h_1,\dots,h_n, r\}$ and $x=\sum_{i=1}^nx_ih_i+ur\in \C$. Formula (\ref{xx}) can be written as
\begin{equation}\label{xxu}
x^2=V(x)=\sum_{j=1}^nx'_jh_j+u'r,
\end{equation}
where the evolution operator $V:x\in\C\to x'=V(x)\in\C$ is defined as the following
\begin{equation}\label{V}
V:\left\{\begin{array}{ll}
x_j'=u\sum_{i=1}^na_{ij}x_i, \ \ j=1,\dots,n,\\[3mm]
u'=u\sum_{i=1}^nb_ix_i.
\end{array}
\right.
\end{equation}
If we write $x^{[k]}$ for the power $(\dots(x^2)^2\dots)$ ($k$ times) with $x^{[0]}=x$ then the trajectory with initial $x$ is given by $k$ times iteration of the operator $V$, i.e. $V^k(x)=x^{[k]}$. This algebraic interpretation of the trajectory is useful to connect powers of an element of the algebra and with the dynamical system generated by the evolution operator $V$. For example, zeros of $V$, i.e. $V(x)=0$ correspond to absolute nilpotent elements of $\C$ and fixed points of $V$, i.e. $V(x)=x$ correspond to idempotent elements of $\C$.

For $x=\sum_{i=1}^nx_ih_i+ur$ define a functional ${\bf b}$ as
$${\bf b}(x)=\sum_{i=1}^nb_ix_i.$$

The following proposition fully describes the set $\mathcal N$ of absolute nilpotent elements of $\C$ with ground field $K=\R$.

\begin{pro}
We have
$$\mathcal N=\{(x,u)\in \C: u=0\}\cup\left\{\begin{array}{ll}
\{(0,\dots,0,u)\in C: u\ne 0\}, \ \ \mbox{if} \ \ \det({\bf A})\ne 0\\[3mm]
\{(x,u)\in \C: u\ne 0, \ \ {\bf A}x=0, \ \ {\bf b}(x)=0\}\ \ \mbox{if} \ \ \det({\bf A})=0,
\end{array}
\right.
$$
where $x=(x_1,\dots,x_n)$, ${\bf A}=(a_{ij})$.
\end{pro}
\begin{proof} An absolute nilpotent element $(x_1,\dots,x_n,u)$ satisfies
\begin{equation}\label{V0}
\left\{\begin{array}{ll}
u\sum_{i=1}^na_{ij}x_i=0, \ \ j=1,\dots,n,\\[3mm]
u\sum_{i=1}^nb_ix_i=0.
\end{array}\right.
\end{equation}
The proof follows from a simple analysis of this system.
\end{proof}
Now we shall describe idempotent elements of $\C$,  these are solutions to $x^2=x$. Such an element
$x=(x_1,\dots,x_n,u)$ satisfies the following
\begin{equation}\label{Vx}
\left\{\begin{array}{ll}
u\sum_{i=1}^na_{ij}x_i=x_i, \ \ j=1,\dots,n,\\[3mm]
u\sum_{i=1}^nb_ix_i=u.
\end{array}\right.
\end{equation}

{\it Case} $u=0$. If $u=0$ then from (\ref{Vx}) we get $x_i=0$ for all $i=1,\dots, n$. Hence $x=0$ is a unique idempotent
element.

 {\it Case} $u\ne 0$. Consider a matrix
$T_u=(t_{ij})_{i,j=1,\dots,n}$ such that
\[
t_{ij}=
\begin{cases}
ua_{ji} & \mbox{if} \ \  i\ne j, \\
ua_{ii}-1 & \mbox{if} \ \   i=j.
\end{cases}
\]
Then first $n$ equations of the system (\ref{Vx}) can be written as
\begin{equation}\label{v5} T_{u}x=0,\end{equation}
where $x=(x_1,\dots,x_n)$.

Consider now $u\in \R\setminus\{0\}$ as a parameter, then equation
(\ref{v5})  has a unique solution $x=0$ if $\det(T_u)\ne 0$,
which gives $u=0$, i.e. this is a contradiction to the assumption
that $u\ne 0$.

If $\det(T_u)= 0$ then we fix a solution $u=u_*\ne 0$ of the
equation $\det(T_u)= 0$. In this case there are infinitely many
solutions $x^*=(x_1^*,\dots,x_n^*)$ of $T_{u_*}x= 0$.
 Substituting
a solution $x^*$ in the last equation of the system (\ref{Vx}), we get
\begin{equation}\label{v6}
{\bf b}(x^*)=\sum_{i=1}^nb_ix^*_i=1.
\end{equation}
Denote by ${\mathcal Id}(\C)$ the set of idempotent elements of $\C$. Hence we have proved the following
\begin{pro}\label{p1} We have
$${\mathcal Id}(\C)=\{0\}\cup
\{(x_1^*,\dots,x_n^*,u_*): u_*\ne 0, \, T_{u_*}x^*=0, \ \
{\bf b}(x^*)=1,\, \det(T_{u_*})=0\}.
$$
\end{pro}

\section{ The enveloping algebra of an EACP}

For a given algebra $\A$ with ground field $K$, we recall that multiplication by elements of $\A$ on the left or on the right give rise to left and right $K$-linear transformations of $\A$ given by $L_a(x)=ax$ and $R_a(x)=xa$. The {\it enveloping algebra}, denoted by $\mathcal E(\A)$, of a non-associative algebra $\A$ is the subalgebra of the full algebra of $K$-endomorphisms of $\A$ which is generated by the left and right multiplication maps of $\A$. This enveloping algebra is necessarily associative, even though $\A$ may be non-associative. In a sense this makes the enveloping algebra ``the smallest associative algebra containing $\A$".

Since an EACP, $\C$, is a commutative algebra the right and left operators coincide, so we use only $L_a$.

\begin{thm} Let $\C$ be an EACP with a natural basis $\{h_1,\dots,h_n,r\}$ and matrix of structural constants $M={\bf A}\oplus{\bf b}$. If $\det({\bf A})\ne 0$
then $\{L_1,\dots,L_n, L_r\}$\, (where $L_i=L_{h_i}$) spans a linear space, denoted by span$(L,\C)$, which is the set
of all operators of left multiplication. The vector
space span$(L,\C)$ and $\C$ have the same dimension.
\end{thm}
\begin{proof} For $x=\sum_{i=1}^nx_ih_i+ur\in \C$ by linearity of multiplication in $\C$ we can write $L_x$ as the following
$$L_x=\sum_{i=1}^nx_iL_i+uL_r.$$
 If $L_x = L_y$, for $y =\sum_{i=1}^ny_ih_i+vr\in \C$, then
$$\left(\sum_{i=1}^nx_ih_i+ur\right)h_j=\left(\sum_{i=1}^ny_ih_i+vr\right)h_j$$
implies $urh_j=vrh_j$, i.e. $u=v$. From
$$\left(\sum_{i=1}^nx_ih_i+ur\right)r=\left(\sum_{i=1}^ny_ih_i+vr\right)r$$
we get
$$\sum_{j=1}^n\left(\sum_{i=1}^n(x_i-y_i)a_{ij}\right)h_j+\left(\sum_{i=1}^n(x_i-y_i)b_{i}\right)r=0.$$

Hence
$$\sum_{i=1}^n(x_i-y_i)a_{ij}=0,\ \ \sum_{i=1}^n(x_i-y_i)b_{i}=0.$$
By assumption $\det({\bf A})\ne 0$ from the last system we get $x_i=y_i$ for all $i=1,\dots,n$.
Thus $x = y$. This means that $L_x$ is an injection.
So the linear space that is spanned by all operators of left multiplication can
be spanned by the set $\{L_i, i=1,\dots,n, r\}$. This set is a basis for span$(L,\C)$.
\end{proof}

\begin{pro} For any $x\in \C$ and any $i, i_1, i_2,\dots,i_m \in \{1,2,\dots,n\}$ the following hold
\begin{equation}\label{L}
L_{i_m}\circ L_{i_{m-1}}\circ\dots\circ L_{i_1}(x)=\left({1\over 2^{m-1}}\prod_{j=1}^{m-1}b_{i_j}\right)L_{i_m}(x),
\end{equation}
\begin{equation}\label{L1}
L_r\circ L_{i}(x)={1\over 2}\sum_{j=1}^na_{ij}L_j(x),
\end{equation}
\begin{equation}\label{L2}
L_i\circ L_r(x)={{\bf b}(x)\over 2}L_i(r).
\end{equation}
\end{pro}
\begin{proof} For $x=\sum_{i=1}^nx_ih_i+ur$ we note that $L_j(x)=uh_jr$, $j=1,\dots,n$.

1) To prove (\ref{L}) we use mathematical induction over $m$. For $m=2$ we have
$$(L_{i_2}\circ L_{i_1})(x)=h_{i_2}(h_{i_1}x)=h_{i_2}(uh_{i_1}r)$$ $$=h_{i_2}{u\over 2}\left(\sum_{i=1}^na_{i_1j}h_j+b_{i_1}r\right)=
(b_{i_1}/2)(uh_{i_2}r)=(b_{i_1}/2)L_{i_2}(x).$$
Assume now that the formula (\ref{L}) is true for $m$, we shall prove it for $m+1$:
$$
L_{i_{m+1}}\circ L_{i_{m}}\circ\dots\circ L_{i_1}(x)=L_{i_{m+1}}\circ\left({1\over 2^{m-1}}\prod_{j=1}^{m-1}b_{i_j}L_{i_m}(x)\right)
$$
$$=\left({1\over 2^{m-1}}\prod_{j=1}^{m-1}b_{i_j}\right)L_{i_{m+1}}\circ L_{i_m}(x)=\left({1\over 2^{m}}\prod_{j=1}^{m}b_{i_j}\right)L_{i_{m+1}}(x).$$

2) Proof of (\ref{L1}):
$$L_r\circ L_i(x)=r(h_ix)=r(uh_ir)={1\over 2}\sum_{j=1}^na_{ij}(urh_j)={1\over 2}\sum_{j=1}^na_{ij}L_j(x).$$

3) Proof of (\ref{L2}):
$$L_i\circ L_r(x)=h_i(rx)=h_i\left(\sum_{j=1}^nx_j(rh_j)\right)$$ $$=
{h_i\over 2}\left(\sum_{m=1}^n\left[\sum_{j=1}^na_{im}x_j\right]h_m+\left(\sum_{j=1}^nx_jb_j\right)r\right)={{\bf b}(x)\over 2}L_i(r).$$
\end{proof}

\section{The centroid of an EACP}

We recall (see \cite{t}) that the centroid $\Gamma(\mathcal A)$ of an algebra $\mathcal A$ is the set of all linear
transformations $T\in \Hom(\mathcal A, \mathcal A)$ that commute with all left and right multiplication
operators
$$TL_x = L_xT, \ \ TR_y = R_yT, \ \ \mbox{for all} \ \ x, y \in \mathcal A.$$

An algebra $\mathcal A$ over a field $K$ is centroidal if $\Gamma(\mathcal A)\cong K$.

\begin{thm} Let $\C$ be an EACP with a natural basis $\{h_1,\dots,h_n,r\}$ and matrix of structural constants $M={\bf A}\oplus{\bf b}$. If $\det({\bf A})\ne 0$ then $\C$ is centroidal.
\end{thm}
\begin{proof} Let $T\in \Gamma(\C)$. Assume
$$T(h_i)=\sum_{j=1}^nt_{ij}h_j+t_ir,  \ \ T(r)=\sum_{k=1}^n\tau_kh_k+\tau r.$$
We have
$$TL_j(h_i)=T(h_jh_i)=0=L_jT(h_i)=h_j\left(\sum_{k=1}^nt_{ik}h_k+t_ir\right)=t_ih_jr.$$
This gives
\begin{equation}\label{T1}
t_i=0, \ \ \mbox{for all} \ \ i=1,\dots,n.
\end{equation}

Now consider
$$TL_j(r)=T(h_jr)={1\over 2}T\left(\sum_{m=1}^na_{jm}h_m+b_jr\right)={1\over 2} \sum_{m=1}^na_{jm}T(h_m)+{b_j\over 2}T(r)$$
$$={1\over 2}\sum_{k=1}^n\left(\sum_{m=1}^na_{jm}t_{mk}+b_j\tau_k\right)h_k+{1\over 2}\left(\sum_{m=1}^na_{jm}t_{m}+b_j\tau\right)r.$$
In another way we have
$$L_jT(r)=h_j\left(\sum_{k=1}^n\tau_kh_k+\tau r\right)=\tau h_jr={\tau\over 2} \left(\sum_{k=1}^na_{jk}h_k+b_j r\right).$$
According to (\ref{T1}) we should have
\begin{equation}\label{T2}
\sum_{m=1}^na_{jm}t_{mk}+b_j\tau_k=\tau a_{jk}.
\end{equation}
Furthermore,
$$TL_r(h_j)={1\over 2}\sum_{k=1}^n\left[\sum_{m=1}^na_{jm}t_{mk}+b_j\tau_k\right]h_k+{1\over 2}\left[\sum_{m=1}^na_{jm}t_{m}+b_j\tau\right]r$$
and
$$L_rT(h_j)= {1\over 2}\sum_{k=1}^n\left[\sum_{m=1}^na_{mk}t_{jm}\right]h_k+{1\over 2}\left[\sum_{m=1}^nt_{jm}b_m\right]r.$$
These equalities imply
 \begin{equation}\label{T3}\begin{array}{ll}
\sum_{m=1}^na_{jm}t_{mk}+b_j\tau_k=\sum_{m=1}^na_{mk}t_{jm}\\[3mm]
\sum_{m=1}^na_{jm}t_{m}+b_j\tau=\sum_{m=1}^nt_{jm}b_m.
\end{array}
\end{equation}
Finally,
$$TL_r(r)=T(rr)=0=L_rT(r)={1\over 2}\sum_{j=1}^n\left(\sum_{k=1}^na_{kj}\tau_k\right)h_j+{1\over 2}\left(\sum_{k=1}^n\tau_kb_k\right)r.$$
Consequently,
\begin{equation}\label{T4}
\sum_{k=1}^na_{kj}\tau_k=0, \ \ j=1,\dots,n; \ \ \sum_{k=1}^nb_k\tau_k=0.
\end{equation}
Since $\det({\bf A})\ne 0$ from (\ref{T4}) we get
\begin{equation}\label{T5}
\tau_i=0, \ \ \mbox{for all} \ \ i=1,\dots,n.
\end{equation}
Using (\ref{T5}), from (\ref{T2}) we obtain
\begin{equation}\label{T6}
\sum_{m=1}^na_{jm}t_{mk}=\tau a_{jk}, \ \ j,k=1,\dots,n.
\end{equation}
Again using $\det({\bf A})\ne 0$, by the Cramer's rule,  from (\ref{T6}) we get the following solution
\begin{equation}\label{T7}
t_{mk}=\left\{\begin{array}{ll}
0, \ \ \mbox{if} \ \ m\ne k\\[3mm]
\tau, \ \ \mbox{if} \ \ m= k.
\end{array}
\right.
\end{equation}
Note that solutions (\ref{T1}), (\ref{T5}) and (\ref{T7}) satisfy system (\ref{T3}). Hence we obtain
 $$T(h_i)=\tau h_i,  \ \ T(r)=\tau r,$$
 where $\tau$ is a scalar in the ground field $K$. That is $T$ is a scalar multiplication. Consequently,
 $\Gamma(\C)\cong K$ and $\C$ is centroidal.\end{proof}

\section{Classification of 2 and 3-dimensional EACP}

Let $\C$ be a $2$-dimensional EACP and $\{h, r\}$ be a basis of this algebra.

It is evident that if $\dim \C^2 =0$ then $\C$ is an  abelian
algebra, i.e. an algebra with all products equal to zero.

\begin{pro}\label{tt}
Any  2-dimensional, non-trivial EACP $\C$ is isomorphic to
one of the following pairwise non isomorphic algebras:
\begin{itemize}
\item[$\C_1$:]\ \ $rh =hr=h$,  \ \ $h^2=r^2 = 0$,
\item[$\C_2$:]\ \ $rh=hr={1\over 2}(h+r)$, \ \ $h^2=r^2 = 0$.
\end{itemize}
\end{pro}
\begin{proof} For an EACP $\C$ we have
$$hr={1\over 2}(ah+br), \ \ h^2=r^2=0.$$

{\it Case:}  $a\ne 0$, $b=0$. By change of basis $h'=h$ and $r'={2\over a}r$ we get the algebra $\C_1$.

{\it Case:}  $a=0$, $b\ne 0$. Take $h'= r$  and $r'={2\over b} h$ then we get the algebra $\C_1$.

{\it Case:}  $a\ne 0$, $b\ne 0$. The change $h'= {1\over a} r$, and $r'={1\over b} h$ implies the algebra $\C_2$.

Since $\C_1^2\C_1^2=0$ and $\C_2^2\C_2^2\ne 0$, the algebras $\C_1$ and $\C_2$ are not isomorphic.
\end{proof}

We note that the algebra $\C_2$ is known as the sex differentiation algebra \cite{m}.

Let now $\C$ be a $3$-dimensional EACP and $\{h_1, h_2, r\}$ be a basis of this algebra.

\begin{thm}\label{t3t}
Any  3-dimensional  EACP $\C$ with dim$(\C^2)=1$ is isomorphic to
one of the following pairwise non isomorphic algebras:
\begin{itemize}
\item[$\C_1$:]\ \ $h_1r=r$;
\item[$\C_2$:]\ \ $h_1r=h_2$;
\item[$\C_3$:]\ \ $h_1r=h_1+r$.
\end{itemize}
In each algebra we take $rh_i=h_ir$, $i=1,2$ and all omitted products are zero.
\end{thm}
\begin{proof} For a 3-dimensional EACP $\C$ we have
$$h_1r=rh_1={1\over 2}(ah_1+bh_2+Ar), \ \ h_2r=rh_2={1\over 2}(ch_1+dh_2+Br), \ \   h_1^2=h_2^2=h_1h_2=r^2=0.$$
First we note that non-zero coefficients of $h_1r$ can be taken 1. Indeed, if $abA\ne 0$ then the change of basis $h_1'={2\over A}h_1$,
$h_2'={2b\over aA}h_2$, $r'={2\over a}r$ makes all coefficients of $h_1r$ equal 1. In case some $a, b, A$ is equal 0 then one can choose a suitable change of basis to make non-zero coefficients equal to 1.
Therefore we have three parametric families: $h_2r=rh_2={1\over 2}(ch_1+dh_2+Br)$ with one of the following conditions
$${\rm (i)} \, h_1r=rh_1=r, \ \ {\rm (ii)}\, h_1r=rh_1=h_2,$$
$${\rm (iii)}\, h_1r=rh_1=h_1+r, \ \ {\rm (iv)}\, h_1r=rh_1=h_2+r,$$
$${\rm (v)}\, h_1r=rh_1=h_1+h_2+r, \ \ {\rm (vi)}\, h_1r=rh_1=h_1,$$
$${\rm (vii)}\, h_1r=rh_1=h_1+h_2.$$

If dim$(C^2)=1$ then $h_2r$ is proportional to $h_1r$. From above-mentioned (i)-(vii) it follows the following cases for $h_1r$ and $h_2r$.

{\it Case} (i): In this case $h_1r=r$ and $h_2r=cr$ for some $c\in K$.
If $c=0$ we get the algebra $\C_1$. If $c\ne 0$ then by the change
$$h_1'=h_1, \ \ h'_2=-h_1+{1\over c}h_2, \ \ r'=r$$ we again obtain the algebra $\C_1$.

{\it Case} (ii): In this case $h_1r=h_2$ and $h_2r=ch_2$ for some $c\in K$.
If $c=0$ we get the algebra $\C_2$. If $c\ne 0$ then by the change
$$h_1'={1\over c}r, \ \ h'_2=ch_1-h_2, \ \ r'=h_2$$ we get the algebra $\C_1$.

{\it Case} (iii): In this case we have $h_1r=h_1+r$ and $h_2r=c(h_1+r)$ for some $c\in K$.
If $c=0$ we get the algebra $\C_3$. If $c\ne 0$ then by the change
$$h_1'=h_1, \ \ h'_2={1\over c}h_2-h_1, \ \ r'=r$$ we get the algebra $\C_3$.

{\it Case} (iv): We have $h_1r=h_2+r$ and $h_2r=c(h_2+r)$ for some $c\in K$.
If $c=0$ then by change
$$h_1'=h_1, \ \ h_2'=h_2, \ \ r'=h_2+r$$
we get the algebra $\C_1$. If $c\ne 0$ then by the change
$$h_1'={1\over c}h_2, \ \ h'_2={1\over c}h_2-h_1, \ \ r'={1\over c}r$$ we get the algebra $\C_3$.

{\it Case} (v): We have $h_1r=h_1+h_2+r$ and $h_2r=c(h_1+h_2+r)$ for some $c\in K$.
 If $c\ne -1$ then by the change
$$h_1'={1\over 1+c}(h_1+h_2), \ \ h'_2={1\over 1+c}(-ch_1+h_2), \ \ r'={1\over 1+c}r$$ we get the algebra $\C_3$.
If $c=-1$ then by the change
$$h_1'=h_1, \ \ h_2'=h_1+h_2, \ \ r'=h_1+h_2+r$$
we get the algebra $\C_1$.

{\it Case} (vi): We have $h_1r=h_1$ and $h_2r=ch_1$ for some $c\in K$.
In this case by the change
$$h_1'=r, \ \ h'_2=ch_1-h_2, \ \ r'=h_1$$ we get the algebra $\C_1$.

{\it Case} (vii): In this case $h_1r=h_1+h_2$ and $h_2r=c(h_1+h_2)$ for some $c\in K$.
Taking the change
$$h_1'={r\over 1+c}, \ \ h'_2=h_2-ch_1, \ \ r'=h_1+h_2$$ we get the algebra $\C_1$.

The obtained algebras are pairwise non-isomorphic this may be
checked by comparison of the algebraic properties listed in the
following table.
\begin{center}
\begin{tabular}{|l|c|c|}
\hline   &  $\C_i^2\C_i^2=0$ & Nilpotent\\
\hline $\C_1$ & Yes & No \\
 \hline $\C_2$ & Yes& Yes \\
\hline $\C_3$ & No & No  \\

\hline
\end{tabular}
\end{center}
\end{proof}
\section*{ Acknowledgements}

 The first author was supported by Ministerio
de Ciencia e Innovaci\'on (European FEDER support included), grant
MTM2009-14464-C02-01. The second author thanks the Department of Algebra, University of
Santiago de Compostela, Spain,  for providing financial support of
his many visits to the Department. He was also supported by the Grant No.0251/GF3 of Education and Science Ministry of Republic
of Kazakhstan. We thank B.A. Omirov for his helpful discussions.

{}
\end{document}